\documentclass[a4paper,12pt]{amsart}
\usepackage{amsmath}
\usepackage{amsthm}
\usepackage{amsfonts}
\usepackage{amssymb,array}
\usepackage{enumerate}
\usepackage{hyperref,graphicx}
\usepackage[all]{xy}

\theoremstyle{definition}
\newtheorem{define}{Definition}[section]
\theoremstyle{plain}
\newtheorem{problem}{Problem}[section]
\theoremstyle{plain}
\newtheorem{assumption}[problem]{Assumption}
\theoremstyle{plain}
\newtheorem{theorem}[define]{Theorem}
\theoremstyle{plain}
\newtheorem{lemma}[define]{Lemma}
\theoremstyle{plain}
\newtheorem{proposition}[define]{Proposition}
\theoremstyle{plain}
\newtheorem{corollary}[define]{Corollary}
\theoremstyle{remark}
\newtheorem{remark}[define]{Remark}
\theoremstyle{definition}

\theoremstyle{definition}

\theoremstyle{definition}

\numberwithin{equation}{section}
\numberwithin{figure}{section}
\numberwithin{table}{section}

\title[Weyl group of semisimple symmetric space]{Weyl group of semisimple symmetric space}
\author[A. Sasaki]{Atsumu SASAKI}

\thanks{This work was supported by JSPS KAKENHI Grant Number JP23K03037. }

\subjclass[2020]{22E46, 22E60, 22F30}
\keywords{semisimple symmetric space; Weyl group; restricted root system; Cartan decomposition}

\address{Department of Mathematics, School of Science, Tokai University, 
4-1-1, Kitakaaname, Hiratsuka, Kanagawa, 259-1292, JAPAN. }
\email{atsumu@tokai.ac.jp}

\date{\today}
\begin{document}

\begin{abstract}
This paper investigates a generalization of the notion of the Weyl group 
of a real reductive Lie group to a semisimple symmetric space $G/H$. 
First, 
we show that the group structure of the Weyl group of $G/H$ is independent of 
the choice of split Cartan subalgebras of $G/H$ 
and the choice of Cartan involutions of $G$ stabilizing $H$. 
After that, 
we can understand all elements of the Weyl group of $G/H$ 
by comparing the Weyl group of some reductive Lie group associated to $G/H$. 
With the aim of extending these results to reductive real spherical homogeneous spaces, 
this work serves as a first step in this direction. 
\end{abstract}

\maketitle

\tableofcontents


\section{Introduction}

Our concern in the present paper is a generalization of the Weyl groups 
of reductive Lie groups to semisimple symmetric spaces. 

Let us begin with our motivation of this study as follows. 
Let $G$ be a non-compact connected real semisimple Lie group 
and $H$ a closed subgroup of $G$ which is a reductive Lie group. 
One can take and fix a Cartan involution $\theta $ of $G$ such that $\theta (H)=H$. 
We set $K:=G^{\theta }$ as the fixed point set of $\theta $ in $G$. 
Then, $K$ is connected and a maximal compact subgroup of $G$. 
Now, 
let us consider the double coset space $K\backslash G/H$. 
This set can be seen from three viewpoints of groups actions 
and these actions could be treated simultaneously by group theoretical method: 
\begin{enumerate}
	\item This is the orbit space of the $(K\times H)$-action on $G$. 
	Here, 
	$K\times H$ acts on $G$ by $(k,h)\cdot x=kxh^{-1}$ ($k\in K,h\in H,x\in G$). 
	\item This is the orbit space of the left action of the compact Lie group $K$ 
	on the homogeneous space $G/H$. 
	\item This is bijective to the orbit space of the left $H$-action 
	on the Riemannian symmetric space $G/K$. 
\end{enumerate}
Concerning these actions, 
let us consider a fundamental problem as follows: 

\begin{problem}
\label{problem:naive}
Find an explicit description of the double coset space $K\backslash G/H$ 
for a reductive homogeneous space $G/H$. 
\end{problem}

In a special case where $H$ coincides with $K$, 
Problem \ref{problem:naive}is nothing but a (classical) Cartan decomposition of $G$. 
The outline of the study for $K\backslash G/H$ is summarized as follows. 
Let us take a maximal split abelian subspace $\mathfrak{a}_0$ for $G$ 
and fix a positive Weyl chamber $\mathfrak{a}_0^+$. 
We denote by $\overline{\mathfrak{a}}_0^+$ the closure of $\mathfrak{a}_0^+$ in $\mathfrak{a}_0$, 
and set $A=\exp \mathfrak{a}_0$ and $\overline{A}_+:=\exp \overline{\mathfrak{a}}_0^+$. 
Then, a Cartan decomposition for $G/K$ means that 
the multiplication map  $K\times \overline{A}_+\times K\to G$ is surjective, 
namely, we obtain a Lie group decomposition $G=K\overline{A}_+K$. 
Moreover, the double coset space is given explicitly by 
\begin{align*}
K \backslash G/K\simeq \overline{A}_+. 
\end{align*}

Next, assume that 
$H$ is non-compact and $G/H$ still has a semisimple (pseudo-Riemannian) symmetric space structure. 
Then, an explicit description of $K\backslash G/H$ is also known as a Cartan decomposition for $G/H$. 
Roughly speaking, 
$K\backslash G/H$ is represented as the closure of a positive Weyl chamber 
of a split Cartan subalgebra for $G/H$. 
See the references \cite[Thm.\,4.1]{fj78} and \cite[Thm.\,10]{ro79}, 
and we will review in Section \ref{subsec:cartan decomposition} on a Cartan decomposition for $G/H$. 

On the other hand, 
when $G/H$ is not a symmetric space, 
it would be difficult to give an answer to Problem \ref{problem:naive}. 
In particular, 
there is no analogue of finding an abelian subspace $\mathfrak{b}_0$ in $\mathfrak{p}_0$ explicitly 
such that the decomposition $G=K(\exp \mathfrak{b}_0)H$ holds 
for a general reductive homogeneous space $G/H$. 

In light of these facts, 
it would be reasonable for the approach to Problem \ref{problem:naive} 
to assume the next situation for a reductive homogeneous space $G/H$. 
Here, 
$\mathfrak{k}_0$ is the Lie algebra of $K$ 
and $\mathfrak{g}_0=\mathfrak{k}_0+\mathfrak{p}_0$ denotes the corresponding Cartan decomposition of 
the Lie algebra $\mathfrak{g}_0$ of $G$. 
For an abelian subspace $\mathfrak{a}_0$ in $\mathfrak{p}_0$, 
we set $A=\exp \mathfrak{a}_0$. 

\begin{assumption}
\label{assump}
A reductive homogeneous space $G/H$ has a Lie group decomposition 
\begin{align}
\label{eq:cd}
G=KAH 
\end{align}
for some abelian subspace $\mathfrak{a}_0$ in $\mathfrak{p}_0$ and $A=\exp \mathfrak{a}_0$. 

We will say that 
the decomposition (\ref{eq:cd}) is a Cartan decomposition for a reductive homogeneous space $G/H$. 
\end{assumption}

If a reductive homogeneous space $G/H$ has an abelian subspace $\mathfrak{a}_0$ in $\mathfrak{p}_0$ 
satisfying Assumption \ref{assump}, then 
the map 
\begin{align}
\label{eq:surjection}
A\to K\backslash G/H,\quad a\mapsto KaH
\end{align}
is a surjection, 
and hence Problem \ref{problem:naive} can be reduced to the following: 

\begin{problem}
\label{problem:cartan decomposition}
Under the setting that $G/H$ satisfies the decomposition (\ref{eq:cd}) 
for some abelian subspace $\mathfrak{a}_0$ in $\mathfrak{p}_0$, 
find a subset $\mathfrak{a}_0'$ in $\mathfrak{a}_0$ 
such that the restriction of the map (\ref{eq:surjection}) 
to $A':=\exp \mathfrak{a}_0'$ is a bijection onto $K\backslash G/H$, namely, $K\backslash G/H\simeq A'$. 
\end{problem}

We are interested in how to find a subset $\mathfrak{a}_0'$ in $\mathfrak{a}_0$ 
such that $K\backslash G/H\simeq \mathfrak{a}_0'$ under the decomposition (\ref{eq:cd}). 
If $G/H$ is a semisimple symmetric space, 
$\mathfrak{a}_0'$ can be taken as the closure $\overline{\mathfrak{a}}_0^+$ 
of the positive Weyl chamber $\mathfrak{a}_0^+$, 
whereas, $\overline{\mathfrak{a}}_0^+$ is also the set of complete representatives of 
the orbit space $W_{K\cap H}(\mathfrak{a}_0)\backslash \mathfrak{a}_0$ 
for the action of the Weyl group $W_{K\cap H}(\mathfrak{a}_0)$ on $A$ 
(see Section \ref{subsec:weyl group associated} for survey). 
Here, $W_{K\cap H}(\mathfrak{a}_0)$ is the quotient group of the normalizer $N_{K\cap H}(\mathfrak{a}_0)$ 
modulo the centralizer $Z_{K\cap H}(\mathfrak{a}_0)$ of $\mathfrak{a}_0$ 
for the adjoint action of $K\cap H$ on $\mathfrak{a}_0$. 
However, 
if we drop off the symmetry of $G/H$, 
then it is difficult to study Problem \ref{problem:naive}. 
Indeed, 
even though $G/H$ has a decomposition (\ref{eq:cd}) for some subspace $\mathfrak{a}_0$ in $\mathfrak{p}_0$, 
we could not expect a bijection between $K\backslash G/H$ and $W_{K\cap H}(\mathfrak{a}_0)$. 

Alternatively, 
let us consider another quotient group 
\begin{align}
\label{eq:def-weyl group}
W_{K\times H}(A):=N_{K\times H}(A)/Z_{K\times H}(A) 
\end{align}
instead of $W_{K\cap H}(\mathfrak{a}_0)$. 
Here, 
$N_{K\times H}(A)$ denotes the normalizer 
and $Z_{K\times H}(A)$ the centralizer of $A=\exp \mathfrak{a}_0$ for the $(K\times H)$-action on $A$. 
Then, one can see that 
$W_{K\times H}(A)$ contains $W_{K\cap H}(\mathfrak{a}_0)$ as a subgroup for general $G/H$, 
and we will see this property in Section \ref{subsec:inclusion} for symmetric $G/H$. 
In this context, 
this seems to be a generalization of the Weyl groups of real reductive Lie groups. 
Henceforth, 
we will say that $W_{K\times H}(A)$ is the {\it Weyl group} of $G/H$ with respect to $K$ and $A$. 

The normalizer $N_{K\times H}(A)$ acts on $A$ as the restriction of the $(K\times H)$-action on $G$, 
and then so does $W_{K\times H}(A)$. 
We write $W_{K\times H}(A)\backslash A$ for the orbit space of the $W_{K\times H}(A)$-action on $A$. 
Then, the map 
\begin{align}
\label{eq:w-surjection}
W_{K\times H}(A)\backslash A\to K\backslash G/H,\quad 
aZ_{K\times H}(A)\mapsto KaH
\end{align}
is still surjective since (\ref{eq:cd}). 
This means that our object $A'$ in Problem \ref{problem:cartan decomposition} 
can be taken as a subset of the representatives of $W_{K\times H}(A)\backslash A$. 

Moreover, 
if one can show that the map (\ref{eq:w-surjection}) is injective, 
then this gives rise to a bijection between $K\backslash G/H$ and $W_{K\times H}(A)\backslash A$. 
Then, one can reduce our problem to find a set of complete representatives of $W_{K\times H}(A)\backslash A$. 
Hence, we raise the next problem: 

\begin{problem}
\label{problem:injective}
Under the setting that $G/H$ satisfies Assumption \ref{assump} 
for some abelian subspace $\mathfrak{a}_0$ in $\mathfrak{p}_0$, 
is the map (\ref{eq:w-surjection}) from $W_{K\times H}(A)\backslash A$ to $K\backslash G/H$ is injective? 
\end{problem}

From now on, 
we will initiate the study of $W_{K\times H}(A)$ given by (\ref{eq:def-weyl group}) 
for a semisimple symmetric space $G/H$. 
The purpose of this paper is to determine $W_{K\times H}(A)$ 
for an arbitrary semisimple symmetric space $G/H$. 
The main result of this paper is explained in Theorem \ref{thm:weyl group}. 

Here is a worth remark that 
the quotient group $W_{K\times H}(A)$ for a compact symmetric triad $(G,K,H)$ has been studied by Matsuki. 
He discovers the structures from the viewpoint of affine Weyl groups. 
Here, a triplet $(G,K,H)$ of compact Lie groups is called a compact symmetric triad 
if both $G/K$ and $G/H$ are compact symmetric spaces. 
See \cite[Thm.\,1]{matsuki97}, \cite{matsuki02} and reference therein. 

This paper is organized as follows. 
First, 
we give a quick review on Cartan decompositions $G=KAH$ for semisimple symmetric spaces $G/H$ 
and an explicit description of the double coset space $K\backslash G/H$ in Section \ref{sec:preliminaries}. 
Next, 
we explain that the group structure of $W_{K\times H}(A)$ defined by (\ref{eq:def-weyl group}) 
for a semisimple symmetric space $G/H$ is determined only by $G/H$ itself in Section \ref{sec:weyl group}. 
More precisely, 
we show that this is independent of both the choice of split Cartan subalgebras of $G/H$ 
and the choice of Cartan involutions of $G$ stabilizing $H$. 
In this context, 
$W_{K\times H}(A)$ could be called the {\it Weyl group of a semisimple symmetric space $G/H$} 
(see Definition \ref{def:weyl group}). 
Finally, we determine the Weyl group $W_{K\times H}(A)$ of a semisimple symmetric space $G/H$. 
For this, 
we characterize the normalizer $N_{K\times H}(A)$, see Theorem \ref{thm:normalizer}. 
Then, our main result (Theorem \ref{thm:weyl group}) is based on Theorem \ref{thm:normalizer}. 
Moreover, 
we give an answer to Problem \ref{problem:injective} for semisimple symmetric spaces, 
see Corollary \ref{cor:problem}. 

{\bf Aknowledgments}. 
The author would like to express his gratitude to Osamu Ikawa and Kurando Baba 
for their interest and valuable comments.


\section{Preliminaries}
\label{sec:preliminaries}

This section gives a brief summary of the facts 
on Cartan decompositions of non-compact semisimple symmetric spaces, 
and futher some preparations for the study of the Weyl groups of real reductive Lie groups. 

\subsection{General setup}
\label{subsec:setting}

We begin with a general setup. 
Let $G$ be a connected non-compact real semisimple Lie group. 
For an automorphism $\nu $ of $G$, 
we write $G^{\nu }:=\{ g\in G:\nu (g)=g\}$ for the fixed point set of $\nu$ in $G$. 
Then, $G^{\nu}$ is a closed subgroup of $G$. 
Let $\sigma$ be an involutive automorphism of $G$ (involution for short) 
and $H$ a closed subgroup of $G$ satisfying $(G^{\sigma })_0\subset H\subset G^{\sigma}$ 
where $(G^{\sigma })_0$ denotes the identity component of $G^{\sigma }$. 
Then, $H$ is reductive and the homogeneous space $G/H$ has a semisimple symmetric space structure. 

It is known that there exists a Cartan involution $\theta $ of $G$ commuting with $\sigma$ (see \cite{be57}). 
We set $K:=G^{\theta }$. 
Then, $K$ is connected and a maximal compact subgroup of $G$. 
As $\sigma \theta =\theta \sigma$, 
the restriction of $\theta $ to the subgroup $H$ is still a Cartan involution of $H$ 
and $H^{\theta }=K\cap H$ is a maximal compact subgroup of $H$. 

If $\sigma$ is conjugate to $\theta $ by the inner automorphism group $\operatorname{Aut}(G)$, 
that is, 
$\sigma$ is also a Cartan involution of $G$, 
then $H$ is isomorphic to $K$, 
namely, 
$H$ is also a maximal compact subgroup of $G$. 
In this case, 
$G/H\simeq G/K$ is a Riemannian symmetric space. 
On the other hand, 
if $\sigma$ is not a Cartan involution of $G$, 
then $H$ is non-compact and hence $G/H$ has a pseudo-Riemannian symmetric space structure. 

Let $\mathfrak{g}_0,\mathfrak{h}_0$ and $\mathfrak{k}_0$ be the Lie algebras of $G,H$ and $K$, respectively. 
For an involution $\nu$ of the Lie group $G$, 
we shall use the same letter $\nu $ 
to denote its differential automorphism on the Lie algebra $\mathfrak{g}_0$. 
Then, $\mathfrak{h}_0$ equals the fixed point set 
$\mathfrak{g}_0^{\sigma}:=\{ X\in \mathfrak{g}_0:\sigma (X)=X\} $ of $\sigma$ in $\mathfrak{g}_0$. 
We set $\mathfrak{q}_0=\mathfrak{g}_0^{-\sigma }=\{ X\in \mathfrak{g}_0:(-\sigma )(X)=X\} $. 
Then, 
$\mathfrak{g}_0=\mathfrak{h}_0+\mathfrak{q}_0$ is the $\sigma$-eigenspace decomposition of $\mathfrak{g}_0$. 
Similarly, $\mathfrak{k}_0$ coincides with $\mathfrak{g}_0^{\theta }$. 
We set $\mathfrak{p}_0=\mathfrak{g}_0^{-\theta }$. 
Then, $\mathfrak{g}_0=\mathfrak{k}_0+\mathfrak{p}_0$ is the corresponding Cartan decomposition 
of $\mathfrak{g}_0$. 

The automorphism $\tau :=\sigma \theta $ is also an involution of $G$. 
Further, $\tau$ is not conjugate to $\sigma$ unless $\sigma$ is a Cartan involution. 
We set $\mathfrak{h}_0^a:=\mathfrak{g}_0^{\tau }$. 
Then, $(\mathfrak{g}_0,\mathfrak{h}_0^a)$ is a semisimple symmetric pair and 
is called the {\it associated symmetric pair} of $(\mathfrak{g}_0,\mathfrak{h}_0)$. 
In this sense, 
we also say that $\mathfrak{h}_0^a$ us an associated Lie algebra for $(\mathfrak{g}_0,\mathfrak{h}_0)$. 
The restriction of $\theta $ to $\mathfrak{h}_0^a$ is still an involution of $\mathfrak{h}_0^a$ 
since $\tau $ commutes with $\theta $. 
Thus, $\mathfrak{h}_0^a$ is decomposed into the $\theta$-eigenspaces as follows: 
\begin{align}
\label{eq:associated algebra}
\mathfrak{h}_0^a
=(\mathfrak{h}_0^a)^{\theta }+(\mathfrak{h}_0^a)^{-\theta }
=\mathfrak{k}_0\cap \mathfrak{h}_0+\mathfrak{p}_0\cap \mathfrak{q}_0. 
\end{align}
In particular, 
$\mathfrak{h}_0^a$ is reductive and 
the semisimple part of $\mathfrak{h}_0^a$, 
which equals the derived ideal $[\mathfrak{h}_0^a,\mathfrak{h}_0^a]$ of $\mathfrak{h}_0^a$, 
is $\theta$-stable. 
Hence, $\theta $ becomes a Cartan involution of $\mathfrak{h}_0^a$, 
and $(\mathfrak{h}_0^a)^{\theta }=\mathfrak{k}_0\cap \mathfrak{h}_0$ 
is a maximal compact subalgebra of $\mathfrak{h}_0^a$. 

Let $H^a$ be a closed subgroup of $G$ with Lie algebra $\mathfrak{h}_0^a$. 
The restriction of $\theta $ to $H^a$ becomes a Cartan involution of $H^a$, 
and then the fixed point set $(H^a)^{\theta }=K\cap H^a$ is a maximal compact subgroup of $H^a$. 
Moreover, 
the Lie algebra of $(H^a)^{\theta }$ equals $\mathfrak{k}_0\cap \mathfrak{h}_0$, 
which coincides with the Lie algebra of $H^{\theta}$. 
Thus, 
we may and do assume that $K\cap H^a$ equals $K\cap H$ by replacing $H^a$ if necessary. 

For convenience, 
we shall say that our choice of $H^a$ is the associated Lie group for $G/H$ 
and also use the notation 
\begin{align*}
K_H:=K\cap H. 
\end{align*}

Let us take a maximal abelian subspace $\mathfrak{a}_0$ in $\mathfrak{p}_0\cap \mathfrak{q}_0$ 
in the following procedure. 
We recall that the subspace $\mathfrak{p}_0\cap \mathfrak{q}_0$ of $\mathfrak{h}_0^a$ 
is expressed as $(\mathfrak{h}_0^a)^{-\theta }$. 
Let $\mathfrak{a}_0^s$ be a maximal abelian subspace 
in $[\mathfrak{h}_0^a,\mathfrak{h}_0^a]^{-\theta }=\mathfrak{p}_0\cap [\mathfrak{h}_0^a,\mathfrak{h}_0^a]$. 
If $\mathfrak{h}_0^a$ is semisimple, 
then the center $\mathfrak{c}(\mathfrak{h}_0^a)$ of $\mathfrak{h}_0^a$ is trivial, 
and then $[\mathfrak{h}_0^a,\mathfrak{h}_0^a]$ coincides with $\mathfrak{h}_0^a$. 
In this case, 
we take $\mathfrak{a}_0$ to be $\mathfrak{a}_0^s$ itself. 
On the other hand, 
if $\mathfrak{h}_0^a$ is not semisimple, 
then the center $\mathfrak{c}(\mathfrak{h}_0^a)$ of $\mathfrak{h}_0^a$ is non-trivial 
and hence 
$\mathfrak{a}_0^s+\mathfrak{c}(\mathfrak{h}_0^a)\cap \mathfrak{p}_0$ 
is a maximal abelian subspace in $\mathfrak{p}_0\cap \mathfrak{q}_0$. 
Thus, we take $\mathfrak{a}_0$ as $\mathfrak{a}_0^s+\mathfrak{c}(\mathfrak{h}_0^a)\cap \mathfrak{p}_0$. 

It is well-known that maximal abelian subspaces in $\mathfrak{p}_0\cap \mathfrak{q}_0$ 
are unique up to conjugation by $K_H$. 
We say that $\mathfrak{a}_0$ is 
a {\it split Cartan subalgebra} of the non-compact semisimple symmetric space $G/H$, 
and the dimension of $\mathfrak{a}_0$ is the {\it split rank} of $G/H$, 
denoted by $\operatorname{rank}_{\mathbb{R}}G/H$. 

\subsection{Group action}
\label{subsec:group action}

In this subsection, 
we set up some group actions which we need in this paper. 

We let the direct product group $G\times G$ act on $G$ by 
\begin{align}
\label{eq:diag-action}
(g_1,g_2)\cdot x:=g_1xg_2^{-1}\quad (g_1,g_2,x\in G). 
\end{align}
Then, 
$K\times H$ acts on $G$ by the restriction of this action 
to the closed subgroup $K\times H$ of $G\times G$. 
Now, let $\Delta$ be the diagonal embedding of $G$ into $G\times G$, namely, 
\begin{align}
\label{eq:delta}
\Delta :G\to G\times G,\quad g\mapsto (g,g). 
\end{align}
Clearly, 
$\Delta $ is a smooth and injective group homomorphism 
and hence $G\simeq \Delta (G)$. 
In this sense, 
we shall regard $G$ as a closed subgroup of $G\times G$. 
Then, the $G$-action on $G$ itself means the restriction of (\ref{eq:diag-action}) to $G$, namely, 
\begin{align}
\label{eq:g-action}
g\cdot x:=\Delta (g)\cdot x=gxg^{-1}\quad (g,x\in G). 
\end{align}
We will use the same letter $\cdot$ to denote the $G$-action on $G$ given by (\ref{eq:g-action}). 

Next, 
we extend the $G$-action on $G$ to the $(G\times G)$-action on $G\times G$, 
namely, we define 
\begin{align*}
(g_1,g_2)*(x_1,x_2):=(g_1\cdot x_1,g_2\cdot x_2)
\quad (g_1,g_2,x_1,x_2\in G). 
\end{align*}
By the restriction of this action to the subgroup $G$, 
let $G$ act on $G\times G$ as follows 
\begin{align}
\label{eq:g-action-diag}
g*(x_1,x_2):=(g\cdot x_1,g\cdot x_2)
\quad (g,x_1,x_2\in G). 
\end{align}

We prepare the next lemma on a relation among the above group actions, 
which we will use in the next section. 

\begin{lemma}
\label{lem:action}
For $g,x_1,x_2,a\in G$, we obtain 
\begin{align*}
(g*(x_1,x_2))\cdot a=g\cdot ((x_1,x_2)\cdot (g^{-1}\cdot a)). 
\end{align*}
\end{lemma}

\begin{proof}
By (\ref{eq:g-action-diag}), 
the element $g*(x_1,x_2)=(gx_1g^{-1},gx_2g^{-1})$ lies in $G\times G$. 
Then, the direct computation shows 
\begin{align*}
(g*(x_1,x_2))\cdot a
&=(gx_1g^{-1})a(gx_2g^{-1})^{-1}
=g(x_1(g^{-1}ag)x_2^{-1})g^{-1}. 
\end{align*}
Since $x_1(g^{-1}ag)x_2^{-1}=(x_1,x_2)\cdot (g^{-1}ag)=(x_1,x_2)\cdot (g^{-1}\cdot a)$, 
we obtain $g(x_1(g^{-1}ag)x_2^{-1})g^{-1}=g\cdot ((x_1,x_2)\cdot (g^{-1}\cdot a))$. 
Hence, we have verified the desired formula. 
\end{proof}

\subsection{Weyl group associated to semisimple symmetric space}
\label{subsec:weyl group associated}

Next, 
we review on the Weyl groups of real reductive Lie groups. 
In particular, 
we see the Weyl group of the associated Lie group $H^a$ of a semisimple symmetric space $G/H$. 

Let us retain the notation and setting as in the previous subsections. 
We set $A:=\exp \mathfrak{a}_0$. 
For the $G$-action on $G$ given by (\ref{eq:g-action}), 
we denote by $N_{K_H}(A)$ and $Z_{K_H}(A)$ 
the normalizer and the centralizer of $A$ in the closed subgroup $K_H=K\cap H$ of $G$, respectively. 
Precisely, 
\begin{align*}
N_{K_H}(A)&:=\{ k\in K_H:k\cdot A=A\} ,\\
Z_{K_H}(A)&:=\{ k\in K_H:k\cdot a=a~(\forall a\in A)\} . 
\end{align*}
Then, $N_{K_H}(A)$ is a closed subgroup of $K_H$ and $Z_{K_H}(A)$ is a normal subgroup of $N_{K_H}(A)$. 
We set the quotient group 
\begin{align}
\label{eq:weyl group of H^a}
W_{K_H}(A):=N_{K_H}(A)/Z_{K_H}(A). 
\end{align}

We recall that 
$K_H$ is a maximal compact subgroup of the real reductive Lie group $H^a$ 
with Lie algebra $\mathfrak{k}_0\cap \mathfrak{h}_0$ 
and $\mathfrak{a}_0$ is a maximal abelian subspace 
in $(\mathfrak{h}_0)^{-\theta }=\mathfrak{p}_0\cap \mathfrak{q}_0$. 
Then, 
$W_{K_H}(A)$ is usually called the {\it Weyl group} of $H^a$. 
It is well-known that 
the Lie algebra of $N_{K_H}(A)$ coincides with the Lie algebra of $Z_{K_H}(A)$. 
Thus, it turns out that $W_{K_H}(A)$ is a finite group. 
It is noteworthy to mention that 
the group structure of $W_{K_H}(A)$ is independent 
of both the choice of split Cartan subalgebras of $G/H$ 
and the choice of Cartan involutions of $G$ commuting with $\sigma$ 
(cf. \cite[Cor.\,2.13 in Chap.\,VII]{helgason}). 
In this context, 
we shall call $W_{K_H}(A)$ the {\it Weyl group associated to $G/H$} in this paper. 

Furthermore, 
the group structure of $W_{K_H}(A)$ is known as follows. 
Let $\Sigma \equiv \Sigma (\mathfrak{h}_0^a,\mathfrak{a}_0)$ be the restricted root system 
of $\mathfrak{h}_0^a$ with respect to $\mathfrak{a}_0$. 
The reflection group $W(\Sigma )$ of $\Sigma$ is called the {\it Weyl group} of $\Sigma$. 
Then, we obtain: 

\begin{proposition}[{cf. \cite[Prop.\,7.32]{knapp}, \cite[Prop.\,7.2.1]{sc84}}]
\label{prop:weylgroup-system}
The Weyl group $W_{K_H}(A)$ associated to a semisimple symmetric space $G/H$ 
(see (\ref{eq:weyl group of H^a})) 
coincides with the Weyl group $W(\Sigma )$ of the restricted root system 
$\Sigma \equiv \Sigma (\mathfrak{h}_0^a,\mathfrak{a}_0)$. 
\end{proposition}

We write $W_{K_H}(A)\backslash A$ for the orbit space of the $W_{K_H}(A)$-action on $A$. 
A set of complete representatives of $W_{K_H}(A)\backslash A$ can be expressed as follows. 
We recall that 
a split Cartan subalgebra $\mathfrak{a}_0$ of $G/H$ is chosen as 
$\mathfrak{a}_0=\mathfrak{a}_0^s+\mathfrak{c}(\mathfrak{h}_0^a)\cap \mathfrak{p}_0$ 
(see Section \ref{subsec:setting}). 
Let us fix a positive system $\Sigma ^+$ of $\Sigma$. 
We take the positive Weyl chamber 
$(\mathfrak{a}_0^s)^+:=\{ X\in \mathfrak{a}_0^s:\lambda (X)>0~(\forall \lambda \in \Sigma ^+)\} $ 
of $\mathfrak{a}_0^s$ related to $\Sigma ^+$ and set 
$\mathfrak{a}_0^+:=(\mathfrak{a}_0^s)^++\mathfrak{c}(\mathfrak{h}_0^a)\cap \mathfrak{p}_0$. 
We denote by $\overline{\mathfrak{a}}_0^+$ the closure of $\mathfrak{a}_0^+$ in $\mathfrak{a}_0$ 
and set $\overline{A}_+:=\exp \overline{\mathfrak{a}}_0^+$. 
Then, we have: 

\begin{proposition}[{cf. \cite[Thm.\,2.22 in Chap.\,VII]{helgason}}]
\label{prop:weyl group action}
$W_{K_H}(A)\backslash A\simeq \overline{A}_+$. 
\end{proposition}

\subsection{Cartan decomposition for semisimple symmetric space}
\label{subsec:cartan decomposition}

In this subsection, 
we summarize a Cartan decomposition for a semisimple symmetric space. 
Now, 
we set $Z_{K_H}(a):=\{ m\in K_H:m\cdot a=a\} $ for $a\in A$. 

%
%

\begin{theorem}[{cf. \cite[Thm.\,4.1]{fj78} and \cite[Prop.\,7.1.3]{sc84}}]
\label{thm:cartan decomposition g/h}
Let $G/H$ be a semisimple symmetric space of a connected non-compact real semisimple Lie group $G$. 
\begin{enumerate}
	\item The multiplication map $\Phi :K\times \overline{A}_+\times H\to G$, $(k,a,h)\mapsto kah$ 
	is surjective. 
	Thus, we get the decomposition of $G$ as follows 
	\begin{align}
	\label{eq:cartan decomposition g/h}
	G=K\overline{A}_+H. 
	\end{align}
	\item For $k\in K$, $a\in \overline{A}_+$ and $h\in H$, 
	the inverse image $\Phi ^{-1}(\{ kah\} )$ is given by 
	$\Phi ^{-1}(\{ kah\} )=\{ (km^{-1},a,mh):m\in Z_{K_H}(a)\} $. 
	Therefore, the map $\overline{A}_+\to K\backslash G/H$, $a\mapsto KaH$ is a bijection, namely, 
	\begin{align}
	\label{eq:double coset a+}
	K\backslash G/H\simeq \overline{A}_+. 
	\end{align}
\end{enumerate}
\end{theorem}

Indeed, a (classical) Cartan decomposition of $G$ 
seems to be a special case $H=K$ of Theorem \ref{thm:cartan decomposition g/h}, 
and then we can say that 
a Cartan decomposition of $G$ is a Cartan decomposition 
for the Riemannian semisimple symmetric space $G/K$. 
In this case, 
$\mathfrak{a}_0$ is taken to be a maximal abelian subspace in $\mathfrak{p}_0$. 

Theorem \ref{thm:cartan decomposition g/h} indicates that 
an arbitrary semisimple symmetric space $G/H$ of non-compact type satisfies Assumption \ref{assump} 
and the answer to Problem \ref{problem:naive} is given by (\ref{eq:double coset a+}). 

Let us see the idea of the proof of Theorem \ref{thm:cartan decomposition g/h} 
according to \cite[Thm.\,4.1]{fj78}. 
The decomposition (\ref{eq:cartan decomposition g/h}) is obtained 
by two decomposition theorems; Mostow decomposition and a (classical) Cartan decomposition of $H^a$. 
Here, Mostow decomposition means that 
the multiplication map 
$K\times \exp (\mathfrak{p}_0\cap \mathfrak{q}_0)\times \exp (\mathfrak{p}_0\cap \mathfrak{h}_0)\to G$ 
is a diffeomorphism \cite{mostow,loos}, 
and a (classical) Cartan decomposition of $H^a$ says $K_H\backslash H^a/K_H\simeq \overline{A}_+$. 
The proof of (\ref{eq:double coset a+}) uses the above two bijections 
and the diffeomorphism $K_H\times \exp (\mathfrak{p}_0\cap \mathfrak{h}_0)\simeq H^a$ 
(this decomposition is also called a Cartan decomposition of $H^a$).

\begin{remark}
Rossmann states in \cite[Thm.\,10]{ro79} 
that the double coset space $K\backslash G/H$ is in bijection with $W_{K_H}(A)\backslash A$. 
\end{remark}

\begin{remark}
If $a\in \overline{A}_+$ is regular, namely, $a\in A_+=\exp \mathfrak{a}_0^+$, 
then the centralizer $Z_{K_H}(a)$ of $a$ in $K_H$ coincides with $Z_{K_H}(A)$ 
(cf. \cite[p.\,117]{sc84}). 
\end{remark}


\section{Weyl group of semisimple symmetric space}
\label{sec:weyl group}

This section will introduce the notion of the Weyl group of a semisimple symmetric space 
and study its group structure. 

Let us recall our setup. 
Let $G$ be a connected non-compact real semisimple Lie group and $\sigma$ be an involution of $G$. 
We take a Cartan involution $\theta $ of $G$ commuting with $\sigma$. 
We set $K=G^{\theta }$ and take a closed subgroup $H$ of $G$ 
such that $(G^{\sigma })_0\subset H\subset G^{\sigma }$. 
We fix a split Cartan subalgebra $\mathfrak{a}_0$ of the semisimple symmetric space $G/H$ 
(see Section \ref{subsec:setting}) 
and set $A=\exp \mathfrak{a}_0$. 

Let $N_{K\times H}(A)$ be the normalizer of $A$ in $K\times H$ 
and $Z_{K\times H}(A)$ the centralizer of $A$ in $K\times H$ 
for the $(K\times H)$-action on $G$ given by (\ref{eq:diag-action}). 
Explicitly, 
\begin{align*}
N_{K\times H}(A)&:=\{ (k,h)\in K\times H:(k,h)\cdot A=A\} ,\\
Z_{K\times H}(A)&:=\{ (k,h)\in K\times H:(k,h)\cdot a=a~(\forall a\in A)\} . 
\end{align*}
Then, 
$N_{K\times H}(A)$ is a closed subgroup of $K\times H$ 
and $Z_{K\times H}(A)$ is a normal subgroup of $N_{K\times H}(A)$. 
The aim of this section is to show that 
the group structures of $N_{K\times H}(A)$ and $Z_{K\times H}(A)$ are invariant 
under the choice of split Cartan subalgebras of $G/H$ 
and the choice of Cartan involutions of $G$ commuting with $\sigma$. 
Then, 
the group structure of the quotient group $N_{K\times H}(A)/Z_{K\times H}(A)$ is determined by $G/H$, 
and hence we introduce the Weyl group of $G/H$ by this quotient group 
(see Definition \ref{def:weyl group}). 

\subsection{Choice of split Cartan subalgebra}
\label{subsec:abelian}

First, 
let us see that 
the group structure of $N_{K\times H}(A)$ (resp. $Z_{K\times H}(A)$) 
is independent of the choice of split Cartan subalgebras of $G/H$. 
For this, 
we fix a Cartan involution $\theta $ of $G$, 
and then fix a maximal compact subgroup $K=G^{\theta }$ of $G$ in this subsection. 

Let us take two split Cartan subalgebras $\mathfrak{a}_0,\mathfrak{b}_0$ of $G/H$. 
Since both $\mathfrak{a}_0$ and $\mathfrak{b}_0$ are maximal abelian subspaces 
in $\mathfrak{p}_0\cap \mathfrak{q}_0$, 
the subspace $\mathfrak{b}_0$ is conjugate to $\mathfrak{a}_0$ by $K_H:=K\cap H$. 
Now, we take $\ell \in K_H$ such as $\mathfrak{b}_0=\operatorname{Ad}(\ell )\mathfrak{a}_0$. 
We set $B=\exp \mathfrak{b}_0$. 
Then, $B$ is conjugate to $A=\exp \mathfrak{a}_0$ by $K_H$ because 
\begin{align}
\label{eq:B}
B=\exp (\operatorname{Ad}(\ell )\mathfrak{a}_0)
=\ell A\ell ^{-1}
=\ell \cdot A. 
\end{align}
Under this setting, 
we get the following relation. 
Here, we recall that 
the notation $*$ means the $G$-action on $G\times G$ given by (\ref{eq:g-action-diag}). 

\begin{proposition}
\label{prop:independent-a}
For $\ell \in K_H$ with (\ref{eq:B}), 
we get $N_{K\times H}(B)=\ell *N_{K\times H}(A)$ and $Z_{K\times H}(B)=\ell *Z_{K\times H}(A)$. 
\end{proposition}

\begin{proof}
Let $(k,h)$ be an element of $N_{K\times H}(B)$. 
Using (\ref{eq:B}), 
the relation $(k,h)\cdot B=B$ means $(k,h)\cdot (\ell \cdot A)=\ell \cdot A$. 
Then, we have 
\begin{align*}
\ell ^{-1}\cdot ((k,h)\cdot (\ell \cdot A))
=\ell ^{-1}\cdot (\ell \cdot A)=A. 
\end{align*}
Here, 
it follows from Lemma \ref{lem:action} that 
$\ell ^{-1}\cdot ((k,h)\cdot (\ell \cdot A))$ equals $(\ell ^{-1}*(k,h))\cdot A$. 
Thus, the element $\ell ^{-1}*(k,h)$ lies in $N_{K\times H}(A)$. 
Hence, $(k,h)$ is an element of $\ell *N_{K\times H}(A)$, 
and then 
$N_{K\times H}(B)$ is contained in $\ell *N_{K\times H}(A)$. 

By the same argument, 
$N_{K\times H}(A)\subset \ell ^{-1}*N_{K\times H}(B)$ is obtained. 
Hence, we conclude 
$N_{K\times H}(B)=\ell *N_{K\times H}(A)$. 

The relation $Z_{K\times H}(B)=\ell *Z_{K\times H}(A)$ follows by a similar argument to the above one. 
Thus, we omit the proof. 
\end{proof}

\subsection{Choice of Cartan involution}
\label{subsec:cartan involution}

Next, we explain that 
the group structure of $N_{K\times H}(A)$ (resp. $Z_{K\times H}(A)$) is independent 
of the choice of Cartan involutions of $G$ commuting with $\sigma$. 
The result of this subsection will be stated in Proposition \ref{prop:independent-cartan}. 

Let us take two Cartan involutions $\theta $, $\theta '$ of $G$ commuting with $\sigma$. 
We still use the same letters $\theta $ and $\theta '$ 
to denote the differential automorphisms of $\theta $ and $\theta '$, respectively. 
It is known from \cite[p.\,153]{loos} 
that there exists $y\in \exp \mathfrak{h}_0$ such that 
\begin{align}
\label{eq:cartan involution}
\theta '=\operatorname{Ad}(y)\circ \theta \circ \operatorname{Ad}(y^{-1})
\end{align}
as an automorphism on the Lie algebra $\mathfrak{g}_0$. 
We set $\mathfrak{k}_0:=\mathfrak{g}_0^{\theta }$, $\mathfrak{p}_0:=\mathfrak{g}_0^{-\theta }$ 
and $\mathfrak{k}_0':=\mathfrak{g}_0^{\theta '}$, $\mathfrak{p}_0':=\mathfrak{g}_0^{-\theta '}$. 

\begin{lemma}
\label{lem:k}
$\mathfrak{k}_0'=\operatorname{Ad}(y)\mathfrak{k}_0$ 
and $\mathfrak{p}_0'=\operatorname{Ad}(y)\mathfrak{p}_0$. 
\end{lemma}

\begin{proof}
Let $X$ be an element of $\mathfrak{k}_0'$. 
As (\ref{eq:cartan involution}), 
the relation $\theta '(X)=X$ implies $\theta (\operatorname{Ad}(y^{-1})X)=\operatorname{Ad}(y^{-1})X$. 
This means that $\operatorname{Ad}(y^{-1})X$ lies in $\mathfrak{k}_0$, 
from which $X\in \operatorname{Ad}(y)\mathfrak{k}_0$. 
Thus, the inclusion $\mathfrak{k}_0'\subset \operatorname{Ad}(y)\mathfrak{k}_0$ is obtained. 
The opposite inclusion also holds by the same argument as above. 
Hence, we get $\mathfrak{k}_0'=\operatorname{Ad}(y)\mathfrak{k}_0$. 

One can also prove the other equality. 
Hence, we omit the proof. 
\end{proof}

The next lemma explains that both $\mathfrak{h}_0=\mathfrak{g}_0^{\sigma }$ 
and $\mathfrak{q}_0=\mathfrak{g}_0^{-\sigma }$ are invariant 
under the automorphism $\operatorname{Ad}(y)$ on $\mathfrak{g}_0$. 

\begin{lemma}
\label{lem:h-q}
$\operatorname{Ad}(y)\mathfrak{h}_0=\mathfrak{h}_0$ and 
$\operatorname{Ad}(y)\mathfrak{q}_0=\mathfrak{q}_0$. 
\end{lemma}

\begin{proof}
We know from \cite[Lem.\,2.3]{bis21} that the involution $\sigma$ commutes with $\operatorname{Ad}(y)$, 
namely, 
\begin{align}
\label{eq:commutative}
\sigma \circ \operatorname{Ad}(y)=\operatorname{Ad}(y)\circ \sigma . 
\end{align}

Let $X$ be an element of $\mathfrak{h}_0$. 
It follows from the equality $\sigma (X)=X$ and the relation (\ref{eq:commutative}) that 
we obtain the relation 
\begin{align*}
\sigma (\operatorname{Ad}(y)X)=\operatorname{Ad}(y)(\sigma (X))=\operatorname{Ad}(y)X. 
\end{align*}
Hence, $\operatorname{Ad}(y)\mathfrak{h}_0$ is contained in $\mathfrak{h}_0$. 
Conversely, 
the equality (\ref{eq:commutative}) implies 
$\sigma \circ \operatorname{Ad}(y^{-1})=\operatorname{Ad}(y^{-1})\circ \sigma $. 
Then, one can also show the inclusion $\operatorname{Ad}(y^{-1})\mathfrak{h}_0\subset \mathfrak{h}_0$, 
from which $\mathfrak{h}_0\subset \operatorname{Ad}(y)\mathfrak{h}_0$. 
Hence, we have verified $\operatorname{Ad}(y)\mathfrak{h}_0=\mathfrak{h}_0$. 

The other equality $\operatorname{Ad}(y)\mathfrak{q}_0=\mathfrak{q}_0$ 
can be also verified by the same argument as the first one, from which we omit the proof. 
\end{proof}

Now, let us take a maximal abelian subspace $\mathfrak{a}_0'$ in $\mathfrak{p}_0'\cap \mathfrak{q}_0$. 
Using $y\in \exp \mathfrak{h}_0$ satisfying (\ref{eq:cartan involution}), 
we set 
\begin{align}
\label{eq:a-a'}
\mathfrak{a}_0:=\operatorname{Ad}(y^{-1})\mathfrak{a}_0'. 
\end{align}
Clearly, $\mathfrak{a}_0$ is an abelian subspace 
in $\operatorname{Ad}(y^{-1})(\mathfrak{p}_0'\cap \mathfrak{q}_0)$. 
By Lemmas \ref{lem:k} and \ref{lem:h-q}, 
the subspace $\operatorname{Ad}(y^{-1})(\mathfrak{p}_0'\cap \mathfrak{q}_0)$ in $\mathfrak{g}_0$ 
coincides with $\mathfrak{p}_0\cap \mathfrak{q}_0$, from which 
$\mathfrak{a}_0$ is contained in $\mathfrak{p}_0\cap \mathfrak{q}_0$. 
Moreover, $\mathfrak{a}_0$ is a maximal abelian subspace in $\mathfrak{p}_0\cap \mathfrak{q}_0$ 
since $\dim \mathfrak{a}_0=\dim \mathfrak{a}_0'=\operatorname{rank}_{\mathbb{R}}G/H$. 

We recall that the exponential map from the Lie algebra $\mathfrak{k}_0$ to 
the connected compact Lie group $K=G^{\theta }$ is surjective, 
and hence $K=\exp \mathfrak{k}_0$ is obtained. 
Similarly, 
$K'=\exp \mathfrak{k}_0'$ is also a maximal compact subgroup of $G$. 
It follows from Lemma \ref{lem:k} that 
\begin{align}
\label{eq:k'-k}
K'=\exp (\operatorname{Ad}(y)\mathfrak{k}_0)
=y(\exp \mathfrak{k}_0)y^{-1}
=y\cdot K. 
\end{align}
Clearly, $y\cdot H$ coincides with $H$. 
Hence, $K'\times H$ is written as 
\begin{align}
\label{eq:k'-h}
K'\times H=(y\cdot K)\times (y\cdot H)=y*(K\times H). 
\end{align}

We set $A'=\exp \mathfrak{a}_0'$ and $A=\exp \mathfrak{a}_0$. 
By (\ref{eq:a-a'}), we obtain 
\begin{align}
\label{eq:a'-a}
A'=\exp (\operatorname{Ad}(y)\mathfrak{a}_0)
=y\cdot A. 
\end{align}

We are ready to consider the normalizer $N_{K'\times H}(A')$ of $A'$ in $K'\times H$. 

\begin{proposition}
\label{prop:independent-cartan}
Retain the setting as above. 
Then, $N_{K'\times H}(A')=y*N_{K\times H}(A)$ and $Z_{K'\times H}(A')=y*Z_{K\times H}(A)$. 
\end{proposition}

\begin{proof}
Let us prove the inclusion $N_{K'\times H}(A')\subset y*N_{K\times H}(A)$. 
Let $(k',h)$ be an element of $N_{K'\times H}(A')$. 
We have already seen in (\ref{eq:k'-h}) that $y^{-1}*(k',h)$ lies in $K\times H$. 

Now, let us verify that $(y^{-1}*(k',h))\cdot a$ is an element of $A$ for any $a\in A$. 
By Lemma \ref{lem:action}, 
we obtain 
\begin{align*}
(y^{-1}*(k',h))\cdot a=y^{-1}\cdot ((k',h)\cdot (y\cdot a)), 
\end{align*}
and $y\cdot a$ is an element of $A'$ (see (\ref{eq:a'-a})). 
As $(k',h)\in N_{K'\times H}(A)$, 
the element $(k',h)\cdot (y\cdot a)$ lies in $A'$. 
Hence, we obtain $y^{-1}\cdot ((k',h)\cdot (y\cdot a))\in y^{-1}\cdot A'=A$, 
from which $(y^{-1}*(k',h))\cdot a\in A$ for any $a\in A$. 
As a consequence, 
$y^{-1}*N_{K'\times H}(A)$ is contained in $N_{K\times H}(A)$, 
and hence the inclusion $N_{K'\times H}(A')\subset y*N_{K\times H}(A)$ has been verified. 

Similarly, we obtain $N_{K\times H}(A)\subset y^{-1}*N_{K'\times H}(A')$. 
This implies the opposite inclusion $y*N_{K\times H}(A)\subset N_{K'\times H}(A')$. 
Therefore, $N_{K'\times H}(A')=y*N_{K\times H}(A)$ has been proved. 

By the same way as in the above argument to the normalizers, 
one can show $Z_{K'\times H}(A')=y*Z_{K\times H}(A)$. 
Thus, we omit the proof.
\end{proof}

\subsection{Weyl group of semisimple symmetric space}
\label{subsec:weyl group}

Propositions \ref{prop:independent-a} and \ref{prop:independent-cartan} explain that 
the group structure of the quotient group $N_{K\times H}(A)/Z_{K\times H}(A)$ 
is independent of both the choice of Cartan involutions of $G$ commuting with $\sigma $ 
and the choice of split Cartan subalgebras of $G/H$. 
Therefore, 
the group structure of $N_{K\times H}(A)/Z_{K\times H}(A)$ is determined by only $G/H$. 

\begin{define}
\label{def:weyl group}
We say that the quotient group 
\begin{align}
\label{eq:weyl group}
W(G/H)\equiv W_{K\times H}(A):=N_{K\times H}(A)/Z_{K\times H}(A). 
\end{align}
is the {\it Weyl group} of a semisimple symmetric space $G/H$. 
\end{define}

\section{Characterization of Weyl group of semisimple symmetric space}
\label{sec:main}

We have seen in Section \ref{sec:preliminaries} that 
an arbitrary semisimple symmetric space $G/H$ of a connected non-compact semisimple Lie group $G$ 
has a Cartan decomposition (\ref{eq:cartan decomposition g/h}) 
and the double coset space $K\backslash G/H$ is given by the closure $\overline{\mathfrak{a}}_0^+$ 
of a positive Weyl chamber $\mathfrak{a}_0^+$ of a split Cartan subalgebra $\mathfrak{a}_0$ 
(see Theorem \ref{thm:cartan decomposition g/h}). 
Furthermore, 
$\overline{A}_+=\exp \overline{\mathfrak{a}}_0^+$ is also in bijection with 
the orbit space $W_{K_H}(A)\backslash A$ of the action of the Weyl group $W_{K_H}(A)$ associated to $G/H$ 
on $A=\exp \mathfrak{a}_0$ (see Proposition \ref{prop:weyl group action}). 

On the other hand, 
we have introduced the Weyl group $W_{K\times H}(A)$ of $G/H$ (see Definition \ref{def:weyl group}) 
in connection with Problem \ref{problem:injective} 
and also seen that the group structure of $W_{K\times H}(A)$ is determined simply by $G/H$. 
Since $G/H$ satisfies (\ref{eq:cd}), 
the map (\ref{eq:w-surjection}) from the orbit space $W_{K\times H}(A)\backslash A$ to $K\backslash G/H$ 
is surjective. 

In view of the relations (\ref{eq:summary}), 
we hope that the orbit space $W_{K\times H}(A)\backslash A$ of the $W_{K\times H}(A)$-action on $A$ 
is also in bijection with $\overline{A}_+$. 
For this, 
we will determine the Weyl group $W_{K\times H}(A)$ of $G/H$ 
by comparing it with the Weyl group $W_{K_H}(A)$ associated to $G/H$. 
\begin{align}
\label{eq:summary}
\begin{array}[t]{c@{\!\!}c@{\!\!}c@{\,\,}c@{\,\,}c@{\!\!}c@{\!\!}c@{\!\!}c@{\!\!}c}
W_{K_H}(A)\backslash A & \simeq & \overline{\mathfrak{a}}_0^+ & \simeq & \overline{A}_+ & \simeq & 
	K\backslash G/H & \leftarrow & W_{K\times H}(A)\backslash A. \\[-1mm]
 & \text{\scriptsize Prop.\,\ref{prop:weyl group action}} & 
 & & 
 & \text{\scriptsize Thm.\,\ref{thm:cartan decomposition g/h}} 
 & 
 & \text{\scriptsize surjection}
\end{array}
\end{align}

\subsection{Inclusion map from $W_{K_H}(A)$ to $W_{K\times H}(A)$}
\label{subsec:inclusion}

We begin with the observation of the diagonal embedding $\Delta$ 
from $G$ to $G\times G$ given by (\ref{eq:delta}). 
As the normalizer $N_{K_H}(A)$ of $A$ in $K_H=K\cap H$ is a closed subgroup of $G$, 
the image $\Delta (N_{K_H}(A))$ of $N_{K_H}(A)$ by $\Delta $ is a closed subgroup of $G\times G$. 
Here, we recall the notation $\ell \cdot A=\ell A\ell ^{-1}=(\ell ,\ell )\cdot A$ 
in the sense of (\ref{eq:diag-action}) and (\ref{eq:g-action}). 
Then, we have: 

\begin{lemma}
\label{lem:inclusion}
$\Delta (N_{K_H}(A))\subset N_{K\times H}(A)$ and $\Delta (Z_{K_H}(A))\subset Z_{K\times H}(A)$. 
\end{lemma}

\begin{proof}
Let $\ell $ be an element of $N_{K_H}(A)$. 
As $\ell \in K_H$, 
the element $\Delta (\ell )=(\ell ,\ell )$ lies in $K_H\times K_H\subset K\times H$. 
Further, the relation 
$\ell \cdot A=A$ implies $(\ell ,\ell )\cdot A=A$, 
from which $\Delta (\ell )$ belongs to $N_{K\times H}(A)$. 
Hence, we have verified $\Delta (N_{K_H}(A))\subset N_{K\times H}(A)$. 

One can show the inclusion $\Delta (Z_{K_H}(A))\subset Z_{K\times H}(A)$ by the same way as above. 
Thus, we omit the proof. 
\end{proof}

Owing to Lemma \ref{lem:inclusion}, 
$N_{K_H}(A)$ and $Z_{K_H}(A)$ are regarded 
as closed subgroups of $N_{K\times H}(A)$ and $Z_{K\times H}(A)$, respectively, 
and 
$\Delta$ induces the group homomorphism $\widetilde{\Delta }$ 
from $W_{K_H}(A)=N_{K_H}(A)/Z_{K_H}(A)$ to $W_{K\times H}(A)=N_{K\times H}(A)/Z_{K\times H}(A)$, 
namely, 
\begin{align}
\label{eq:tilde-delta}
\widetilde{\Delta }:W_{K_H}(A)\to W_{K\times H}(A),\quad 
[\ell ]_{K_H}\mapsto [\Delta (\ell )]_{K\times H}=[\ell ,\ell ]_{K\times H}. 
\end{align}
Here, 
$[\ell ]_{K_H}\in W_{K_H}(A)$ denotes the equivalence class of $\ell \in N_{K_H}(A)$ modulo $Z_{K_H}(A)$ 
and $[k,h]_{K\times H}\in W_{K\times H}(A)$ for the equivalence class of $(k,h)\in N_{K\times H}(A)$ 
modulo $Z_{K\times H}(A)$. 

\begin{lemma}
\label{lem:injection}
The group homomorphism $\widetilde{\Delta }$ given by (\ref{eq:tilde-delta}) is injective. 
\end{lemma}

\begin{proof}
Suppose that two elements $\ell ,\ell '\in N_{K_H}(A)$ satisfy 
$\widetilde{\Delta }([\ell ]_{K_H})=\widetilde{\Delta }([\ell ']_{K_H})$. 
By definition, 
we obtain $[\ell ,\ell ]_{K\times H}=[\ell ',\ell ']_{K\times H}$. 
Now, we set 
\begin{align}
\label{eq:ell-ell'}
(k,h):=(\ell ,\ell )^{-1}(\ell ',\ell '). 
\end{align}
This relation implies $k=\ell ^{-1}\ell '=h$, 
and hence $k$ lies in $K\cap H=K_H$. 
On the other hand, 
$(k,h)$ is an element of $Z_{K\times H}(A)$. 
Then, the relation $(k,h)\cdot a=a$ holds for any $a\in A$. 
Thus, we obtain $k\cdot a=kak^{-1}=kah^{-1}=a$ for any $a\in A$. 
This means that $k$ is an element of $Z_{K_H}(A)$. 
Therefore, we get $[\ell ']_{K_H}=[\ell k]_{K_H}=[\ell ]_{K_H}$. 
\end{proof}

By Lemma \ref{lem:injection}, 
$W_{K_H}(A)$ is regarded as a closed subgroup of $W_{K\times H}(A)$ in the following sense 
\begin{align}
\label{eq:inclusion}
W_{K_H}(A)\simeq \widetilde{\Delta }(W_{K_H}(A))\subset W_{K\times H}(A). 
\end{align}

\subsection{Determination of $N_{K\times H}(A)$}
\label{subsec:key}

In view of the inclusion (\ref{eq:inclusion}), 
we raise a natural question 
whether the opposite inclusion also holds or not. 
To give an answer to this question, 
we again consider the restriction of $\Delta $ to $N_{K_H}(A)$, 
which gives rise to the injective homomorphism to $N_{K\times H}(A)$. 

\begin{theorem}
\label{thm:normalizer}
Retain the setting as in Section \ref{subsec:inclusion}. 
Then, the subgroup $\Delta (N_{K_H}(A))$ coincides with $N_{K\times H}(A)$. 
Hence, we get the group isomorphism $N_{K\times H}(A)\simeq N_{K_H}(A)$. 
Similarly, $Z_{K\times H}(A)\simeq Z_{K_H}(A)$. 
\end{theorem}

\begin{proof}
Let us take an element $(k,h)$ of $N_{K\times H}(A)$ and put 
\begin{align}
\label{eq:a-def}
a:=(k,h)\cdot e=keh^{-1}=kh^{-1}\in A. 
\end{align}
Due to Proposition \ref{prop:weyl group action}, 
one can take $\ell _+\in N_{K_H}(A)$ such that $\ell _+\cdot a$ lies in $\overline{A}_+$. 
We write 
\begin{align*}
a_+:=\ell _+\cdot a. 
\end{align*}

It is obvious that $a_+=(e,e)\cdot a_+$ belongs to the $(K\times H)$-orbit in $G$ through $a_+$. 
On the other hand, 
$a_+$ is also expressed as 
\begin{align*}
a_+
=(\ell _+,\ell _+)\cdot ((k,h)\cdot e)
=(\ell _+k,\ell _+h)\cdot e. 
\end{align*}
Since $\ell _+\in K_H$, $k\in K$ and $h\in H$, 
the element $\ell _+k$ lies in $K$ and $\ell _+h$ in $H$. 
Hence, $a_+$ belongs to the $(K\times H)$-orbit in $G$ through $e\in \overline{A}_+$, too. 
Thus, we get 
\begin{align}
\label{eq:orbit-eq}
KeH=Ka_+H. 
\end{align}
Since $K\backslash G/H\simeq \overline{A}_+$ (see Theorem \ref{thm:cartan decomposition g/h}), 
the relation (\ref{eq:orbit-eq}) implies that 
the element $a_+\in \overline{A}_+$ must be equal to $e$, and then $\ell _+\cdot a=e$. 
Hence, we get $a=e$. 
Since the element $a$ is defined by (\ref{eq:a-def}), 
we obtain $k=h$. 

As $(k,h)\in N_{K\times H}(A)$, 
we obtain $A=(k,h)\cdot A=(k,k)\cdot A=k\cdot A$. 
This means that $k$ belongs to $N_{K_H}(A)$. 

Consequently, 
an element $(k,h)$ of $N_{K\times H}(A)$ is expressed as $(k,k)=\Delta (k)\in \Delta (N_{K_H}(A))$. 
This means that 
the injective homomorphism $\Delta :N_{K_H}(A)\to N_{K\times H}(A)$ is surjective, 
and hence $\Delta (N_{K_H}(A))=N_{K\times H}(A)$. 
Therefore, we have proved $N_{K_H}(A)\simeq N_{K\times H}(A)$. 

One can show the group isomorphism $Z_{K\times H}(A)\simeq Z_{K_H}(A)$ 
by the same argument as the proof of $N_{K\times H}(A)\simeq N_{K_H}(A)$. 
Thus, we omit the proof. 
\end{proof}

\subsection{Determination of $W_{K\times H}(A)$}
\label{subsec:main}

We have seen in Lemma \ref{lem:injection} that 
the group homomorphism $\widetilde{\Delta }$ given by (\ref{eq:tilde-delta}) is injective. 
On the other hand, since 
\begin{align*}
N_{K_H}(A)\stackrel{\Delta }{\simeq }N_{K\times H}(A)\to W_{K\times H}(A),\quad 
k\mapsto [\Delta (k)]_{K\times H}
\end{align*}
is surjective, 
so is $\widetilde{\Delta }$. 
Therefore, we conclude: 

\begin{theorem}
\label{thm:weyl group}
Let $G$ be a connected non-compact real semisimple Lie group 
and $H$ a closed subgroup of $G$ such that $G/H$ is a semisimple symmetric space. 
Then, the Weyl group $W_{K\times H}(A)$ of $G/H$ coincides with the image $\widetilde{\Delta }(W_{K_H}(A))$ 
of the Weyl group $W_{K_H}(A)$ associated to $G/H$. 
In particular, we get 
\begin{align*}
W_{K\times H}(A)\simeq W_{K_H}(A). 
\end{align*}
\end{theorem}

Theorem \ref{thm:weyl group} also asserts that 
the group structure of $W_{K\times H}(A)$ is the same as that of $W_{K_H}(A)$. 
More precisely, we obtain: 

\begin{corollary}
\label{cor:reflection group}
The Weyl group $W_{K\times H}(A)$ of a semisimple symmetric space $G/H$ 
is isomorphic to the reflection group $W(\Sigma )$ 
of the restricted root system $\Sigma \equiv \Sigma (\mathfrak{h}_0^a,\mathfrak{a}_0)$ 
of the associated Lie algebra $\mathfrak{h}_0^a$ with respect to $\mathfrak{a}_0$ as a group. 
Hence, $W_{K\times H}(A)$ is a finite group. 
\end{corollary}

\begin{proof}
The group isomorphism $W_{K\times H}(A)\simeq W_{K_H}(A)$ follows from Theorem \ref{thm:weyl group}. 
On the other hand, 
the Weyl group $W_{K_H}(A)$ associated to $G/H$ is isomorphic to $W(\Sigma )$ as a group 
(see Proposition \ref{prop:weylgroup-system}). 
Hence, this corollary has been verified. 
\end{proof}

\begin{remark}
It is well-known that 
the Lie algebra of $N_{K_H}(A)$ coincides with the normalizer 
$N_{\mathfrak{k}_0\cap \mathfrak{h}_0}(\mathfrak{a}_0)$ of $\mathfrak{a}_0$ 
in $\mathfrak{k}_0\cap \mathfrak{h}_0$ 
and coincides with the Lie algebra $Z_{\mathfrak{k}_0\cap \mathfrak{h}_0}(\mathfrak{a}_0)$ of $Z_{K_H}(A)$. 
Then, the finiteness of $W_{K_H}(A)$ follows from this fact. 
In contrast, 
it does not seem easy to understand whether 
the Lie algebra $\operatorname{Lie}(N_{K\times H}(A))$ of $N_{K\times H}(A)$ 
coincides with the Lie algebra of $Z_{K\times H}(A)$. 
In the case of semisimple symmetric spaces, 
Theorem \ref{thm:normalizer} proves 
that $\operatorname{Lie}(N_{K\times H}(A))$ equals $N_{\mathfrak{k}_0\cap \mathfrak{h}_0}(\mathfrak{a}_0)$. 
\end{remark}

\subsection{Answer to Problem \ref{problem:injective} for semisimple symmetric space}
\label{subsec:answer}

Finally, 
we give an answer to Problem \ref{problem:injective} for semisimple symmetric spaces 
by comparing the known facts with the new result. 

\begin{corollary}
\label{cor:problem}
For a semisimple symmetric space $G/H$, 
we have a bijection $K\backslash G/H\simeq W_{K\times H}(A)\backslash A$. 
\end{corollary}

\begin{proof}
In view of the relations (\ref{eq:summary}), 
it is sufficient for the proof to verify the bijection 
between two orbit spaces $W_{K\times H}(A)\backslash A$ and $W_{K_H}(A)\backslash A$. 

We recall that $W_{K_H}(A)$ acts on $A$ as follows 
\begin{align*}
[\ell ]_{K_H}\cdot a=\ell \cdot a=\Delta (\ell )\cdot a\quad 
(\ell \in N_{K_H}(A),a\in A). 
\end{align*}
On the other hand, 
it follows from Theorem \ref{thm:weyl group} that an arbitrary element of $W_{K\times H}(A)$ 
is of the form $\widetilde{\Delta }([\ell ]_{K_H})=[\Delta (\ell )]_{K\times H}$ 
for some $\ell \in N_{K_H}(A)$. 
Then, the $W_{K\times H}(A)$-action on $A$ is written as 
\begin{align*}
[\Delta (\ell )]_{K\times H}\cdot a=\Delta (\ell )\cdot a\quad 
(\ell \in N_{K_H}(A),a\in A). 
\end{align*}
Thus, $[\Delta (\ell )]_{K\times H}\cdot a=[\ell ]_{K_H}\cdot a$ is obtained 
for any $\ell \in N_{K_H}(A)$ and any $a\in A$, 
and hence $W_{K\times H}(A)\cdot a=W_{K_H}\cdot a$ for any $a\in A$. 
Therefore, $W_{K_H}(A)\backslash A$ is in bijection with $W_{K\times H}(A)\backslash A$. 
\end{proof}

\begin{remark}
When $G$ is a reductive Lie group such that its semisimple part is compact 
(in particular, $G$ itself is compact) 
and both $(G,K)$ and $(G,H)$ are compact symmetric pairs, 
Hoogenboom \cite{ho83} and Matsuki \cite[Thm.\,1]{matsuki97} show that 
there is a decomposition $G=KAH$ for some torus $A$ and 
Matsuki \cite[Thm.\,1]{matsuki97} (see also \cite{matsuki02}) proves 
the bijection $K\backslash G/H\simeq W_{K\times H}(A)\backslash A$ 
for the similar quotient group $W_{K\times H}(A)$ as in Definition \ref{def:weyl group} 
in this compact setting 
(he uses the notation $J$ instead of $W_{K\times H}(A)$). 
The idea of the proof of this bijection discussed in \cite[Thm.\,1]{matsuki97} can be applied 
to the proof of Corollary \ref{cor:problem} in our non-compact setting. 
\end{remark}

The study of Problem \ref{problem:injective} is going to the next stage 
for reductive real spherical homogeneous spaces 
which are a wider class of semisimple symmetric spaces and contain non-symmetric ones. 
Indeed, 
reductive real spherical homogeneous spaces satisfy the decomposition (\ref{eq:cd}) 
as a generalization of a Cartan decomposition (cf. \cite{sa10,sa15,sa25} and \cite{ta22}). 
Recently, 
we explicitly determine the group structure of the Weyl group $W_{K\times H}(A)$ 
for a class of reductive real spherical homogeneous spaces 
which are realized as one-dimensional fiber bundles over semisimple symmetric spaces, 
which will be discussed in a forthcoming paper \cite{sa-pre}. 




\begin{thebibliography}{99}

\bibitem{bis21}
K. Baba, O. Ikawa and A. Sasaki, 
A duality between non-compact semisimple symmetric pairs and commutative compact semisimple symmetric triads 
and its general theory, 
{\it Differential Geom. Appl.} {\bf 76} (2021), Paper No. 101751, 42 pp. 

\bibitem{be57}
M. Berger, 
Les espaces symm\'etriques noncompacts, 
{\it Ann. Sci. \'Ecole Norm. Sup. } {\bf 74} (1957), 85--177. 

\bibitem{fj78}
M. Flensted-Jensen, 
Spherical functions of a real semisimple Lie group. A method of reduction to the complex case, 
{\it J. Functional Analysis} {\bf 30} (1978), 106--146. 


\bibitem{ho83}
B. Hoogenboom, 
Intertwining functions on compact Lie groups, 
Ph.D. thesis, Mathematisch Centrum, Amsterdam, 1983. 

\bibitem{helgason}
S. Helgason, 
{\it Differential geometry, Lie groups, and symmetric spaces}, Graduate Studies in Mathematics, 34, 
American Mathematical Society, Providence, RI, 2001. 

\bibitem{knapp}
A. W. Knapp, 
{\it Lie groups beyond an introduction}, 
Progr. Math., {\bf 140}, Birkh\"auser Boston, Inc., Boston, MA, 2002. 


\bibitem{loos}
O. Loos, 
Symmetric spaces. I: General theory, 
W. A. Benjamin, Inc., New York-Amsterdam 1969. 

\bibitem{matsuki97}
T. Matsuki, 
Double coset decompositions of reductive Lie groups arising from two involutions, 
{\it J. Algebra} {\bf 197} (1997), 49--91. 

\bibitem{matsuki02}
T. Matsuki, 
Classification of two involutions on compact semisimple Lie groups and root systems, 
{\it J. Lie Theory} {\bf 12} (2002), 41--68. 

\bibitem{mostow}
G. D. Mostow, 
Some new decomposition theorems for semisimple Lie groups, 
{\it Mem. Amer. Math. Soc.} {\bf 14} (1955), 31--54. 

\bibitem{ro79}
W. Rossmann, 
The structure of semisimple symmetric spaces, 
{\it Canad. J. Math.} {\bf 31} (1979), 157--180. 

\bibitem{sa10}
A. Sasaki, 
A characterization of non-tube type Hermitian symmetric spaces by visible actions, 
{\it Geom. Dedicata} {\bf 145} (2010), 151--158. 

\bibitem{sa15}
A. Sasaki, 
Admissible representations, multiplicity-free representations and visible actions on
non-tube type Hermitian symmetric spaces, 
{\it Proc. Japan Acad. Ser. A Math Sci.} {\bf 91} (2015), 70--75. 

\bibitem{sa25}
A. Sasaki, 
Invariant measures on non-symmetric reductive real spherical homogeneous spaces of rank-one type, 
{\it Symmetry in Geometry and Analysis, Vol. 1, Festscrhift for Toshiyuki Kobayashi}, 
497--531, Progr. Math. {\bf 357}, Birkh\"auser/Springer, Singapore, 2025. 

\bibitem{sa-pre}
A. Sasaki, 
Weyl group for reductive real spherical homogeneous spaces of one-dimensional fiber bundle type, 
{\it in preparation}. 

\bibitem{sc84}
H. Schlichtkrull, 
Hyperfunctions and harmonic analysis on symmetric spaces, 
Progr. Math. {\bf 49}, Birkh\"auser Boston, Inc., Boston, MA, 1984. 

\bibitem{ta22}
Y. Tanaka, 
A Cartan decomposition for a reductive real spherical homogeneous spaces, 
{\it Kyoto J. Math.} {\bf 62} (2022), 95--102. 

\end{thebibliography}
\end{document}